\documentclass{amsart} 
\usepackage{amssymb}
\usepackage{amsfonts}
\usepackage{amsopn}
\usepackage{epsfig}
\usepackage{amsmath}
\usepackage{epic,eepic,latexsym}

\DeclareMathOperator{\Pic}{Pic}

\DeclareMathOperator{\rank}{rank}

\DeclareMathOperator{\sgn}{sgn}

\DeclareMathOperator{\Td}{Td}

\let\?=\overline
\let\:=\colon
\let\bb=\mathbb

\let\rar=\rightarrow

\let\s=\mathcal
\let\onto=\twoheadrightarrow

\let\wt=\widetilde

\theoremstyle{plain}
 \newtheorem{thm}{Theorem}[section] 
 \newtheorem{cor}[thm]{Corollary}
 \newtheorem{lem}[thm]{Lemma} 
 \newtheorem{prp}[thm]{Proposition}

\newtheoremstyle{sbs}
  {\smallskipamount}
  {\smallskipamount}
  {\it}
  {\parindent}
  {\rm}
  {.}
  {.5em}
  {}
\theoremstyle{sbs}

\theoremstyle{definition}
 \newtheorem{dfn}[thm]{Definition}
 
 \newtheorem{eg}[thm]{Example}

 \newtheorem{rmk}[thm]{Remark}

\makeatletter
 \@addtoreset{equation}{thm}
\makeatother

\def\mylistparam
 {\setlength{\topsep}{\smallskipamount}%
  \setlength{\itemsep}{0pt}\setlength{\parsep}{0pt}%
  \setlength{\itemindent}{0pt}%
  \setlength{\leftmargin}{\parindent}%
  \setbox0=\hbox{1.\kern0.5em}\addtolength{\leftmargin}{\wd0}%
 }%
 {\begin{list}{\llap{\rm (\roman{enumi})}}%
  {\usecounter{enumi}\mylistparam}%
  }%
 {\end{list}}

\begin{document}

\title
  {Brill-Noether-type Theorems with a Movable Ramification Point} 

\author[R. C. Lehman]{Rebecca C. Lehman}
 \address
 {Mathematics Department, Tulane University\\
 6823 St. Charles Ave.\\
 New Orleans, LA 70118, USA}
 \email{rlehman@tulane.edu}

\begin{abstract}
The classical Brill-Noether theorems count the dimension of the family of maps from a general curve of genus $g$ to non-degenerate curves of degree $d$ in projective space $\bb{P}^r.$  These theorems can be extended to include ramification conditions at fixed general points.  This paper deals with the problem of imposing a ramification condition at an unspecified point.  We solve the problem completely in dimensions $1$ and $2,$ and provide an existence test and bound the dimension of the family in the general case.
\end{abstract}

\maketitle

\section{Introduction}

In their seminal paper on algebraic functions and their geometric applications (\cite{B-N}, 1879), Brill and Noether calculated the expected dimension $\rho$ of the family of maps from a general curve of genus $g$ to $\bb{P}^r.$  However, they did not prove that the family has dimension exactly $\rho,$ or even that any such maps exist.  The existence theorem was first proved with twentieth-century rigor by Kleiman and Laksov (\cite{KL-S}, 1972; \cite{KL-C}, 1974) and independently by Kempf (\cite{Kf}, 1971).  The non-existence and dimensionality results were proved by Griffiths and Harris (\cite{GH-BN}, 1980) and refined by Eisenbud and Harris (\cite{EH-LS}, 1986, \cite{EH-BN}, 1986).  Griffiths, Harris and Eisenbud's proofs extend almost verbatim to the case when one imposes in addition the condition that the linear system must have a specified type of ramification at a general fixed point $P$ of the curve.  But this raises the more basic question of whether a $g^r_d$ exists with the specified ramification at any point at all.  In particular, let $X$ be a general curve of genus $g,$ and let positive integers $r,$ $d$ and $(m_0 ,\hdots , m_r)$ be given.  Does there exist a $g^r_d$ on $X$ possessing vanishing sequence $(m_0, \hdots, m_r)$ at any point $Q$?  If so, what is the dimension of the set of such pairs $(\s{L}, Q)$?  If the dimension is zero, how many actual pairs are there?  

Based on the classical Brill-Noether theorems and the theorems for fixed general ramification points, one might be led to make the na\"ive conjecture that the family of $g^r_d$'s with a given ramification type exists if and only if the expected dimension is greater than or equal to zero, and has the expected dimension.  This paper will show that the na"ive statement holds perfectly when $r=1.$  When $r \geq 2$ the existence can fail, but it can be decided by an explicit numerical criterion (Theorem \ref{Wrd}).  The dimension is exact when $r=1$ or $r=2.$  For $r\geq 2$ we prove a weak general bound on the dimension.

Section \ref{Prel} will provide some preliminary definitions and notation for reference.  Section \ref{GP} contains a proof of the Existence Criterion, by calculating the intersection class of the family of such maps in terms of Theta classes on the Picard variety.  If the intersection class is nonzero, then the locus must be non-empty.  The family $W^r_d$ is a degeneracy locus of a morphism of filtered vector bundles on the product $\Pic^d_C \times C,$ and its class can be expanded as a sum of determinants.  In Section \ref{LLS} we prove the dimensionality theorem for $r=1$ (Theorem \ref{r1}) and the weak general bound (Theorem \ref{WkBd}).  We can degenerate any linear systems on the general curve to limit linear systems on a special reducible curve $X_0,$ which will be a flag curve consisting only of rational and elliptic components.  The inequalities that define limit linear series will allow us to determine the possible limit linear series explicitly on each component and bound their dimension.  By the upper semicontinuity property, the dimension on the limit $X_0$ is greater than or equal to the dimension on the general curve, so we obtain upper bounds on the general curve.  Section \ref{PC} uses the extrinsic geometry of curves in $\bb{P}^2$ with fixed points to complete the dimensionality proof for $r=2$ (Theorem \ref{r2}).  All degree $d$ curves form a projective space of dimension $n=\frac{1}{2}d(d+3).$  All point conditions cut out linear subspaces of this projective space, so up to $n$ of them always impose independent conditions.  We shall see that when $\rho<0,$ the dimension of the subscheme of plane curves of degree $d$ and genus $g$ that satisfy the given ramification condition is strictly less than the number of degrees of freedom allowed by the moduli space of curves of genus $g,$ the automorphisms of $\bb{P}^2,$ and the ramification point and the nodes.  Hence the general curve does not admit a map to such a plane curve.

While the nonexistence proof for $r=2$ appears ad hoc, it does demonstrate that there are ramified limit linear systems on the flag curve that do not extend to linear systems on the general curve.  The flag curve is not ''sufficiently general'' with respect to Brill-Noether conditions with moving ramification points, although it is ''sufficiently general'' to detect such conditions without ramification or with only fixed ramification.  So the failure of the degeneration proof should not be construed as strong evidence against the dimensionality conjecture in general. 

\section{Preliminaries}\label{Prel}

We begin with a smooth, connected, projective curve $C$ of genus $g$ over the complex numbers $\bb{C}.$  This section will provide a reference for definitions and important lemmas.

\subsection{Linear systems and their parameter spaces}

\begin{dfn}
A linear system of degree $d$ and dimension $r+1,$ or $g^r_d,$ on $C$, is an $(r+1)$-dimensional vector space of linearly equivalent divisors on $C.$  
\end{dfn}

It will be helpful to use both additive and multiplicative notation.  Multiplicatively, a $g^r_d$ can be given as a pair $(\s{L}, V),$ where $\s{L}$ is a line bundle on $C$ and $V$ is an $(r+1)$-dimensional subspace of $H^0(\s{L}).$  A basis of $V$ will be denoted by $\sigma_0, \cdots, \sigma_r.$  Additively, a $g^r_d$ will be given as a vector space $L$ of linearly equivalent divisors on $C,$ with basis $D_0, \cdots, D_r.$  If $L$ is base-point-free, that is, if there is no point $P$ contained in every divisor in $L,$ then $L$ determines a map $\phi_L$ of degree $d$ from the curve $C$ to projective space $\bb{P}^r$ up to projective equivalence.  So a $g^r_d$ can be given equivalently by the pair $(\s{L}, V),$ by $L,$ or by a base divisor $B$ of degree $b \leq d$ and a map $\phi_{L-B}\: C \rar \bb{P}^r$ of degree $d-b.$  By abuse of notation we shall use these notations interchangeably without further comment.

\begin{dfn}
Let $(\s{L}, V)$ be a $g^r_d$ on $C$, and let $P$ be a point on $C.$  An \emph{order basis} for $V$ at $P$ is a basis $(\sigma_0, \cdots, \sigma_r)$ of $V$ constructed as follows: Given $(\sigma_0, \cdots, \sigma_j),$ take $\sigma_{j+1}$ to be any section linearly independent of $(\sigma_0, \cdots, \sigma_j)$ that vanishes to the highest possible order at $P.$
\end{dfn}

Any two order bases $\sigma_i$ and $\tau_i$ at $P,$ differ by a transformation of the form $\tau_i=\sum_{j=0}^i c_j \sigma_j,$ i.e. by triangular matrices.  

\begin{dfn}
The \emph{vanishing sequence} or \emph{multiplicity sequence} $(m_0(V,P), \cdots, m_r(V, P))$ of a $g^r_d$ $(\s{L},V)$ at a point $P$ is given by the orders of vanishing $v_P(\sigma_i)$ of the elements of an order basis at $P.$
\end{dfn}

In particular, nonzero multiplicity $m_r$ indicates that the point $P$ is a base point of the linear system, a nonzero $m_{r-1}$ indicates that the image of $P$ under $\phi$ is an $(m_{r-1})$-fold multiple point, and each $m_i$ indicates an osculating linear subspace of codimension $i+1.$

\begin{dfn}
The \emph{ramification sequence} $(a_0, \cdots, a_r)$ of $(\s{L},V)$ at $P$ is given by $a_i=m_i-(r-i).$
\end{dfn}
Note that we have ordered the vanishing sequence and the multiplicity sequence from greatest to least.  This is the reverse of the customary ordering, but it makes no difference for the degeneration arguments, and it will simplify the filtrations in Section \ref{GP}.

\begin{dfn}
The \emph{weight} or \emph{total weight} of $\s{L}$ at $P,$ is the sum $w(\s{L},P)=\sum_{i=0}^r a_i.$  It will be denoted $w(P)$ when $\s{L}$ is understood.
\end{dfn}

Let $\Pic^d_C$ be the Picard scheme of line bundles of degree $d.$  The Chow ring of $\Pic^d_C$ is generated by the divisor $\theta,$ which is the image in $\Pic^d_C$ of the $(g-1)$-fold product of $C$ with itself, and satisfies the relation $\theta^g=g![*],$ where $[*]$ is the point class, and $\theta^{g+1}=0.$

Let $W^r_d$ be the locus in $\Pic^d_C$ consisting of line bundles $\s{L}$ with at least $r+1$ global sections, and let $W^r_d(m_0, \cdots, m_r)$ be the locus of line bundles $\s{L}$ with at least $r+1$ global sections vanishing to orders at least $m_0, \cdots, m_r$ at some point $Q.$  Let $\s{P}_d$ be a Poincar\'e sheaf on $\Pic^d_C \times C,$ and let $\s{E}$ be the pushforward of $\s{P}_d$ to $\Pic^d.$  Let $\s{G}^r_d$ be the Grassmann bundle $\bb{G}(r+1, \s{E})$ over $\Pic^d_C$ whose fiber over a point $[\s{L}]$ is the set of $(r+1)$-dimensional subspaces of $H^0(\s{L}).$  Then the Chow ring of $\s{G}^r_d$ is generated by the pullback of $\theta$ and by the Schubert classes $\sigma_{i_1, \cdots, i_m}$.  Let $\s{G}^r_d(m_0, \cdots, m_r)$ denote the subscheme of pairs $(\s{L}, V)$ such that $V$ has a basis of sections vanishing to orders $m_0, \cdots, m_r$ at \emph{some} point $Q$ on $C.$

Let $\s{M}_g$ be the moduli space of smooth curves of genus $g,$ and $\overline{\s{M}_g}$ be its natural compactification, the moduli space of stable curves of genus $g.$  These moduli spaces have dimension $1$ when $g=1,$ and $3g-3$ for $g \geq 2.$  Let $\Delta_g$ be the boundary divisor $\overline{\s{M}_g}-\s{M}_g.$  

\subsection{Degeneration and limit series}

It will be convenient to consider families of curves in $\overline{\s{M}_g}$ that degenerate to the boundary.  Let $T \rar \overline{\s{M}_g}$ be a one-parameter family with universal curve $X\rar T,$ such that the generic geometric fiber $X_{\overline{\eta}}$ is a smooth irreducible curve, whereas the special fiber $X_0$ is a reduced but reducible curve of compact type.  Let $(\s{L}, V)$ be a $g^r_d$ on $X_{\overline{\eta}}.$  
After a finite base change, we may assume that the sheaf $\s{L}$ is defined on $X_{\eta}.$  After blowing up if necessary, we may assume from now on that the ramification points of $\s{L}$ specialize to smooth points of $X_0.$

Since the total space $X$ is smooth, $\s{L}$ extends to a sheaf on $X.$  That extension, however, is not unique: we can vary it by twisting by a divisor supported on $X_0$.  If $\wt{\s{L}}$ is an extension of $\s{L}$ and $D$ is any divisor of $X$ supported on $X_0,$ then $\wt{\s{L}} \otimes \s{O}_X(D)$ is another.  Fortunately this is the only ambiguity: if $\wt{\s{L}}$ and $\wt{\s{L}}'$ are any two extensions of $\s{L},$ then $\wt{\s{L}} \otimes \wt{\s{L}}'^{-1}$ is trivial away from $X_0,$ so it must be the line bundle associated to some divisor $D$ supported on $X_0.$  The total degree of any extension $\wt{\s{L}}$ of $\s{L}$ is $d.$  So the sum of the degrees $\wt{\s{L}}_Y$ over all components $Y$ of $X_0$ is $d.$  Since $X_0$ is of compact type and the intersection pairing on the components of $X_0$ is unimodular, there exists an extension $\s{L}_Y$ of $\s{L}$ whose degree is $d$ on $Y$ and $0$ on all other components.

While no one of these extensions is more canonical than the others, together they are unique and determine $(\s{L}, V).$ 

\begin{dfn}
A \emph{limit linear series} on a reducible curve $X_0$ is an association to each component $Y$ of $X_0$ a $g^r_d$ $(\s{L}_Y, V_Y),$ called the $Y$-aspect, satisfying the \emph{Compatibility Condition}: For any two components $Y$ and $Z$ of $X$ meeting at a point $P,$ and for any $i,$ $$0 \leq i \leq r,$$ then $$m_i(V_Y, P)+m_{r-i}(V_Z,P)=d.$$
\end{dfn}

Every linear series $X_\eta$ gives rise to a distinct limit linear series, but the converse need not be true: there may be limit $g^r_d$'s that do not arise from a $g^r_d$ on the smooth fibers.

The following are the key properties of limit series, due to Eisenbud and Harris:

\begin{lem}[\cite{EH-GP}, Prop. 1.3] 
Let $X_0$ be a reduced but reducible curve of compact type, let $Y$ and $Z$ be irreducible components of $X_0$ meeting at $P,$ and let $P'$ be another point of $Y.$  Let $(\s{L}, V)$ be a limit linear series on $X_0.$ Then the multiplicities satisfy the inequality 
\begin{equation}\label{P-P'}
m_i(V_Y, P') \leq m_i(V_Z, P).
\end{equation}
\end{lem}

\begin{lem}[\cite{EH-GP}, Prop. 1.5]\label{rat-cusp}
Let $Y$ be a rational component of $X_0.$  Let $P$ be the intersection between $Y$ and a component of positive genus, or between $Y$ and a chain of rational curves $W_j^k$ terminating in a curve of positive genus.  Then the aspect $V_Y$ has at least a cusp at $P$. 
\end{lem}

\begin{cor}[\cite{EH-GP}, Cor. 1.6]\label{rat-P-P'}
Let $Y$ be a rational component of $X_0.$  Let $Q$ be the intersection of $Y$ with a chain of $W_j^i$'s terminating in a curve of positive genus.
\begin{itemize}
\item Let $P$ and $P'$ be any two points of $Y$ not equal to $Q.$  Then there is at most one section of $V$ vanishing only at $P$ and $P'.$
\item If $Y$ meets another component $Z$ at $P,$ then $m_i(V_Y, P') < m_i(V_Z,P)$ for all but at most $1$ value of $i.$ 
\end{itemize}
\end{cor}

\begin{prp}[Additivity For General Reducible Curves, \cite{EH-LS}, 4.5]\label{add}
Let $X$ be a curve of compact type whose components are general curves $X_1, \cdots, X_c$ of genus $g_1, \cdots, g_c.$  Let $P_1, \cdots, P_s$ be a set of general points on $X_1, \cdots, X_c,$ \emph{and, let the nodes of $X$ be general points on the components.}  Then the dimension of the family of $g^r_d$'s on $X$ with specified multiplicities $m_i(P_j)$ is exactly equal to $$\rho=g-(r+1)(g+r-d)-\sum_{j=1}^s \sum_{i=0}^r (m_i(P_j)-i).$$
\end{prp}

\subsection{Plane curves}

Plane curves of degree $d$ are defined by homogeneous equations of degree $d$ in $3$ variables, modulo scalars.  The space of all homogeneous equations of degree $d$ is an affine space spanned by the set of monomials of degree $d,$ so it has dimension $\frac{1}{2}(d+1)(d-2).$  Modulo scalars, we obtain a projective space of dimension $N=\frac{1}{2}d(d-3).$

The simplest condition we can impose on a plane curve is that it should pass through a given point $P,$ with multiplicity $m.$  This corresponds to the linear subspace of codimension $\frac{1}{2}m(m+1)$ in $\bb{P}^N$ cut out by the equations requiring $f$ to vanish along with all its partial derivatives up to $m$.  Let $P$ be a point on a plane curve.  If we blow up the plane at $P,$ we obtain an exceptional divisor.  Points on the exceptional divisor are called "virtual points," or "infinitely near points, of $P.$  If $C$ has at least an $m$-fold point at $P$, then we can require $C$ to have an $m$-fold point at $P',$ and obtain an additional $\frac{1}{2} m(m+1)$ conditions.  All point conditions, whether actual or virtual, impose independent conditions on $d$-forms up to codimension $N.$

A smooth plane curve of degree $d$ has genus $\frac{1}{2}(d-1)(d-2).$  An ordinary $r$-fold point on $C$ drops the genus by $\frac{1}{2}r(r-1),$ and a singular point $P$ drops the genus by $\sum_{P'} \frac{1}{2}r(r-1),$ where the sum is over all infinitely near points $P'$ in the neighborhood of $P.$

A $g^2_d$ determines a map from an abstract curve $C$ to the plane up to a change of coordinates, given by a $3x3$ matrix up to scalars.

\section{Existence Results}\label{GP}

The idea behind these enumerative existence proofs is that the cycle class of an empty set must be zero.  If we can compute the class of the locus of $g^r_d$'s with a given property and show that it is nonzero, then such $g^r_d$'s must exist.  We do this by expressing the locus as the degeneracy locus of an appropriate map of vector bundles.

We first consider the case when $g+r-d>0.$

Pull back the problem to $\Pic^d_C \times C \times C,$ using the second copy of $C$ to parameterize the moving point $Q.$  Let $\Delta$ be the diagonal on $C \times C.$  Pull back the Poincar\'{e} sheaf $\s{P}_d$ to $\Pic^d_C \times C \times C$ by $\pi_{12}^*.$  The fiber of the vector bundle $\pi_{12}^*(\s{P}_d(nP))/(\pi_{12}^*\s{P}_d)(-m_i \Delta)$ over a point $Q$ of the second copy of $C$ is just $\s{P}_d(nP)/\s{P}_d(-m_iQ).$  

So we can consider the map of vector bundles $$\s{E} \rar \s{F}_0 \rar \s{F}_1 \rar \cdots \rar \s{F}_r$$ on 
$\Pic^d_C \times C,$ where $$\s{F}_i=\pi_{12*}(\pi_{12}^*(\s{P}_d(nP))/(\pi_{12}^*\s{P}_d)(-m_i\Delta)).$$  From now on we shall suppress the $\pi_{12}$ for ease of notation.  Over the point $([\s{L}], Q)$ this reduces to the map $\s{L}(nP)\rar \s{L}(nP)/\s{L}(-m_i Q).$  We want to calculate the degeneracy locus $W$ on $\Pic^d_C \times C$ where each of these maps has kernel of dimension at least $i.$  Then the Gysin image of this locus on $\Pic^d_C$ will be our class $W^r_d(m_i).$  We compute the total Chern classes of $\s{E}$ and $\s{F}_i.$

\begin{lem} The total Chern class of the vector bundle $\s{E}$ is $e^{-\theta}.$
\end{lem}

\begin{proof}
Applying the K\"unneth decomposition to $\Pic^d \times C,$ write $c(\s{P}_d(nP))=1+(d+n)\zeta+\gamma,$ where $\zeta$ is the pullback of the point class from $C,$ and $\gamma$ is the class of the intersection pairing on $H^1(C)$ and $H^1(\Pic^d_C).$  Note that $\zeta^2=\zeta\gamma=0,$ and $\gamma^2=-2\theta\zeta.$  Expand the Chern character as $$ch(\s{P}_d(nP))=e^{c_1(\s{P}_d(nP))}=\sum_{k\geq 0} \frac{((d+n)\zeta + \gamma)^k}{k!}=1+(d+n)\zeta+\gamma +\frac{1}{2}\gamma^2,$$ since all higher terms vanish.  To calculate $c(\s{E}),$ apply Grothendieck-Riemann-Roch.  The Todd class of the vertical tangent bundle is the pullback of the Todd class of the curve $C,$ which is $1-\frac{1}{2} \omega_C=1+(1-g)\zeta.$  Hence $$ch(\s{E})=\pi_{1*}(\Td(T^v)ch(\s{P}_d(nP)))=\pi_{1*}\left(\left(1+(1-g)\zeta\right)\left(1+(d+n)\zeta+\gamma - \theta \zeta\right)\right).$$

The Gysin image $\pi_{1*}$ takes the coefficient of $\zeta$ in the sum, which in our case is $1+d+n-g-\theta.$  So $ch(\s{E})=1+d+n-g-\theta.$  Hence $$c_k(\s{E})=((-1)^k\theta^k/k!),$$ so the total Chern class is $c(\s{E})=e^{-\theta}.$
\end{proof}

\begin{lem}
The total Chern class of $\s{F}_i$ is $$1+(d+(g-1)(m_i-1))\zeta+m_i\gamma - m_i(m_i-1)\zeta\theta.$$
\end{lem}

\begin{proof}
We can filter $\s{P}(nP)/\s{P}(-m_i\Delta)$ with successive quotient bundles of the form $\s{P}(kP)/\s{P}((k-1)P)$ and $\s{P}(-k\Delta)/\s{P}(-(k+1)\Delta).$  The former terms are trivial.  The latter can be written as $\s{P} \otimes \omega_C^{\otimes k},$ and we have $c(\s{P})=1+d\zeta+\gamma$ on $\Pic^d_C \times C.$  Since the diagonal $\Delta$ is another degree-$1$ copy of $C$ in $\Pic^d_C \times C \times C,$ pulling back $\s{P}_d$ to $\Pic^d_C \times C$ and restricting to $\Pic^d_C \times \Delta$ gives the same Chern class $1+d\zeta + \gamma.$  Since $c(\omega_C)=1+(2g-2)\zeta,$ where $\zeta$ is the pullback of the point class from $C,$ $$c(\s{P}_d \otimes \omega_C^{\otimes k})=1+(d+2k(g-1))\zeta + \gamma.$$  Hence the class $c(\s{P}_d(nP)/\s{P}_d(-m_i\Delta))$ is the product $$\prod_{k=0}^{m_i-1}(1+(d+2k(g-1))\zeta + \gamma)=1+m_i \left(d+(g-1)(m_i-1)\right) \zeta + m_i \gamma +\frac{(m_i-1)m_i}{2}\gamma^2,$$ which we can rewrite as $$1+m_i(d+(m_i-1)(g-1))\zeta+ m_i\gamma - m_i(m_i-1)\zeta\theta.$$
\end{proof}

To compute the degeneracy locus $W^r_d,$ we use Fulton's generalized Thom-Porteous formula for maps of filtered vector bundles (\cite{F-D}, Thm. 10.1), to express the class of the degeneracy locus of a map of filtered bundles as a determinant in the Chern classes of the bundles.

\begin{lem}[\cite{F-D}, Thm. 10.1]\label{F-P}
Suppose we are given partial flags of vector bundles $$A_0 \subseteq \hdots \subseteq A_{k-1}$$ and $$B_0 \onto \hdots \onto B_{k-1}$$ on a scheme $X,$ of ranks $a_0 \leq \hdots \leq a_{k-1}$ and $b_0 \geq \hdots \geq b_{k-1},$ and a morphism $h\: A_{k-1} \rar B_0$. (Note that equalities are allowed in these bundles.)

Let $r_0, \ldots, r_{k-1}$ be nonnegative integers satisfying $$0 < a_0-r_0 < \hdots < a_{k-1}-r_{k-1},$$ $$b_0-r_0 > \hdots > b_{k-1}-r_{k-1} >0,$$  Define $\Omega$ to be the subscheme defined by the conditions that the rank of the map from $A_i$ to $B_i$ is at most $r_i$ for $0 \leq i \leq k-1.$  Let $\mu$ be the partition $(q_0^{n_0}, \ldots, q_{k-1}^{n_{k-1}}),$ where
$q_i=b_i-r_i,$ and $n_0=a_0-r_0, \, n_i=(a_i-r_i)-(a_{i-1}-r_{i-1})$ for $0 \leq i \leq k-1.$  Let $n=a_{k-1}-r_{k-1}.$  For $0 \leq i \leq n-1,$ let $\rho(i)=\min\{s\in [0, k-1]: i \leq a_s-r_s=n_0+\hdots + n_s\}.$

The class of $\Omega$ in $A(X)$ is $P_r \cap [X],$ where $$P_r=\det(c_{\mu_i-i+j}(B_{\rho(i)}-A_{\rho(i)}))_{0\leq i, j \leq n-1}.$$  
\end{lem}
In our case, the ranks of the two vector bundles are $$a_i=\rank \s{E}=d+n-g$$ for all $i,$ and $$b_i=\rank \s{F}_i=n+m_i.$$  We want to impose the rank conditions $r_i=d+n-g-i.$  Hence $a_i-r_i=i,$ so the sequence of $a_i-r_i$ is strictly increasing.  For all $i\geq 1,$ we have $$n_i=i-(i-1)=1,$$ so $\rho_i=i$ for all $i.$

Suppose first that $m_i-m_{i+1} \geq 2$ for all $i.$  Then since $b_i=n+m_i$ and $$b_i-r_i=m_i+g-d+i,$$ the sequence $b_i-r_i$ is strictly decreasing, so Theorem \ref{F-P} applies directly.  Hence $$\mu_i=m_i+i+g-d.$$  So the class of $W$ is $c_{m_i+j+g-d}(\s{F}_i-\s{E}).$ Otherwise, suppose that we have a sequence of $l+1$ multiplicities decreasing by $1,$ say $m_k, m_{k+1}, \cdots, m_j=m_k-l.$  There is redundancy in requiring all the multiplicity conditions.  The condition that at most $j$ basis elements vanish at the point $Q$ to multiplicity at least $m_j$ implies all the others.  We can forget about $\s{F}_{k}, \cdots, \s{F}_j$ altogether and renumber the indices to omit it.  Hence for $i>k,$ we have $r_i=d+n-g-i-l,$ so $n_i=1$ for all $i$ except $i=k,$ where $n_i=l.$  So $q_k$ is to be repeated $l+1$ times.  Hence the sequence $\mu_i$ which counts the $q_i$ with their multiplicities, is unchanged.  We still have $$\mu_i=m_i+i+g-d$$ for $0 \leq i \leq r.$  However, $\s{F}_k$ has now been replaced by $\s{F}_{k}.$ 

Hence we have
\begin{lem}
The class of $W$ on $\Pic^d_C$ is $$\det\left(c_{m_i+g-d+j}\left(e^\theta \cdot \big(1+m'_i(d+(m'_i-1)(g-1))\zeta +m'_i \gamma - m'_i(m'_i-1)\zeta \theta\big)\right)\right)_{0\leq i, j \leq n-1}$$ where $m'_k$ is the greatest value $m_j \leq m_k$ such that $m_{j+1} < m_j -1,$
\end{lem}

We expand out the determinant as  $\det(a_{ij}),$ where 
\[a_{ij}=\frac{\theta^{m_i+g-d+j}}{(m_i+g-d+j)!}+\zeta\theta^{m_i+g-d-1+j}\frac{m'_i(d+(m'_i-1)(g-1))}{(m_i+g-d-1+j)!} \]\[-\zeta\theta^{m_i+g-d-1+j}\frac{(m'_i-1)m'_i}{(m_i+g-d-2+j)!}+\gamma\theta^{m_i+g-d-1+j}\frac{m'_i}{(m_i+g-d-1+j)!}
\]
We can break up this matrix as a sum.  Set $M_{ij}=\frac{\theta^{m_i+g-d+j}}{(m_i+g-d+j)!}.$ This is the classical term that exists without the movable ramification point.   All but one or two components of the product will be of this form.  Set $N_{ij}=\zeta\theta^{m_i+g-d-2+j}\frac{(m'_i)(d+(m'_i-1)(g-1))}{(m_i+g-d-2+j)!}.$  This term comes from the $\zeta$ part of the canonical sheaf $\omega_C.$  It is always positive.  Since it contains $\zeta,$ it is killed by multiplication with any other term containing $\zeta$ or $\gamma.$  

Set $L_{ij}=\zeta\theta^{m_i+g-d-1+j}\frac{(m'_i-1)m'_i}{(m_i+g-d-2+j)!}.$  This term comes from the $\gamma^2$ in $c(\s{F}_i),$ so it is subtracted.  It contains $\zeta,$ so it is killed by any other term containing $\zeta$ or $\gamma.$
Finally, set $G_{ij}=\gamma\theta^{m_i+g-d-1+j}\frac{m'_i}{(m_i+g-d-1+j)!}.$  This term contains $\gamma$ instead of $\zeta,$ it is killed by multiplication by anything containing $\zeta$ or $\theta^{g-1}.$

We want to expand the determinant $\det(M_{ij}+N_{ij}-L_{ij}+G_{ij})$ as a sum.
\begin{lem}
If $C=A+B,$ then $$\det (C)= \sum_{S \subset \{1, 2, \cdots, r+1\}} \det (D_{ij}(S)),$$ where $D_{ij}(S)=A_{ij}$ if $i\in S,$ otherwise $B_{ij}.$
\end{lem}

\begin{proof}
It follows immediately from expanding out the definition of the determinant, $$\det(C)=\sum_{\sigma} \sgn(\sigma)\prod_{i=1}^{r+1}(A_{i\sigma(i)}+B_{i\sigma(i)}).$$ 
\end{proof}

In our case, almost all the terms vanish when we expand the determinant $\det(M_{ij}+N_{ij}-L_{ij}+G_{ij})$ and we are left with $X+Y+Z,$  where the first term is
$X=\sum_{k=1}^{r+1} \det(X_{ij}(k)),$
 where $$X_{ij}(k)=\left \{\begin{array}{ccc}
M_{ij} &\text{ if } &i \neq k\\
N_{kj} &\text{ if } & i=k
\end{array}\right\},$$

the second term is
$Y=- \sum_{k=1}^{r+1} \det(Y_{ij}(k)),$ where $$Y_{ij}(k)
\left \{ \begin{array}{ccc}
M_{ij} & \text{ if } & i \neq k\\
L_{kj} & \text{ if } & i=k
\end{array}\right\},$$ 

and the third term is
$Z= \sum_{1 \leq k \leq l \leq r+1} \det(Z_{ij}(k,l)),$ where $$Z_{ij}(k,l)=\left\{\begin{array}{ccccc}M_{ij}& \text{ if } &i \neq k& \text{ and } &i\neq l\\
G_{ij} & \text{ if }& i=k &\text{ or } &i=l
\end{array}\right\}.$$

All the other possible combinations of $M,$ $N,$ $L$ and $G$ vanish because they contain $\zeta^2,$ $\zeta\gamma,$ or else they fail to contain $\zeta,$ so their Gysin images vanish on $\Pic^d_C.$

Expanding the three terms using the Vandermonde formulas, summing them and taking the Gysin image on $\Pic^d_C,$ we finally obtain the following 
\begin{thm}\label{Wrd}
Let $X$ be a general curve of genus $g$ and let $r, d, m_i$ be numbers such that $g+r-d \geq 0$ and $$\rho(g,r,d,m_i) \geq 0.$$  Then the class $[W^r_d(m_0,\cdots, m_r)]$ of the family of $g^r_d$'s admitting a point $Q$ with vanishing sequence $m_i$ is given by 
\small(\[ W^r_d(m_i)=\theta^c \sum_{k=0}^{r}\frac{(m'_k)(m_k+g+r-d)}{\prod_{i=0}^{r}(m_i+g+r-d)!}\Bigg[\big((d+(m'_k-1)(g-1))\]
\[\prod_{i>j, \, i\neq k, \, j\neq k}(m_i-m_j)\prod_{i \neq k}|m_i-m_k+1|\big)\]
\[- \big(m'_k-1)(m_k+g+r-d-1)\prod_{i>j, \, i\neq k, \, j\neq k}(m_i-m_j)\prod_{i<k}(m_i-m_k+2)\prod_{i>k}(m_k-m_i-2)\big)\]
\[-\sum_{l \neq k}(m'_l)(m_l+g+r-d)|m_k-m_l|\big(\prod_{i> j, \, i \neq k, \, i \neq l, \, j \neq k, \, j \neq l}(m_i-m_j)\prod_{i \neq k}|m_i-m_k+1||m_i-m_l+1|\big)\Bigg],\] where $$c=\sum_{i=0}^r (m_i-i+g+d-r).$$
\end{thm}

\begin{eg}
A canonical curve has exactly $(g-1)(g)(g+1)$ ramification points.
\end{eg}

\begin{proof}
A canonical curve has $d=2g-2,$ $r=g-1,$ and the ramification must be at least $(g, g-2, g-3, \hdots, 1, 0).$  For any $k \neq 0, k\neq g,$ there exists $i=k+1$ such that $m_i-m_k+1=0.$  For any $k\neq g,$ there exists $i=k+2$ such that $m_i-m_k+2=0.$  When $k=g,$ the coefficient $m_g$ is zero.  So it is sufficient to consider the first term 
\[ \theta^s \sum_{k=0}^r\frac{(m'_k)(m_k+g+r-d)}{\prod_{i=0}^{r}(m_i+g+r-d)!}(d+(m'_k-1)(g-1))
\prod_{i>j, \, i\neq k, \, j\neq k}(m_i-m_j)\prod_{i \neq k}|m_i-m_k+1|\] for $k=0;$ $m'_k=g.$  Note that the codimension is $s=g.$

We have \small{\[\theta^g \frac{(g)(g+g+(g-1)-(2g-2))}{\prod_{i=0}^{g-1}(m_i+g+(g-1)-(2g-2))!}((2g-2)+(g-1)(g-1))
\prod_{i> j \neq g}(m_i-m_j)\prod_{i \neq k}|g-m_i-1|\]
\[= \zeta\theta^g \frac{(g)(g+1)}{\prod_{i=0}^{g-2}(i+1)!(g+1)!}(g+1)(g-1)
\prod_{g-1> i > j \geq 0}(i-j)\prod_{g-1>i>0}|g-i-1|\]
\[=\zeta \theta^g \frac{g(g+1)}{\prod_{i=0}^{g-1}i!(g+1)!}(g+1)(g-1)\prod_{i=0}^{g-1}i! (g-2)!]\]}
Since $\theta^g=g!,$ we have $\theta^g(g+1)=(g+1)!.$  This cancels the $(g+1)!$ in the denominator.  The products $\prod_{i=1}^{g-1} i!$ in the numerator and the denominator cancel with each other, leaving $g(g+1)(g-1).$
\end{proof}

\begin{eg} 
For any $g, r, d$ such that $g+r-d\geq 0,$ and $$g-(r+1)(g+r-d)\geq 0,$$ the expected class of $g^r_d$'s possessing a point with the simplest possible ramification $(r+1, r-1, \cdots, 0)$ is positive. 
\end{eg}

\begin{proof}
The class is positive because the negative terms all contain factors of $$m'_k \prod|m_i-m_k+2|$$ and $$m'_k\prod|m_i-m_k+1||m_i-m_l+1|,$$ so they vanish.

If the dimension of $W^r_d$ on $\Pic^d_C$ is $\rho(g,r,d),$ the dimension of pairs $(\s{L},Q)$ on $\Pic^d_C \times C$ with $\s{L} \in W^r_d$ is $\rho + 1.$  The ramification imposes one additional condition, so the expected dimension of pairs $(\s{L},Q)$ with simple ramification at $Q$ is $\rho.$  Since the coefficient of the class is positive, the locus is non-empty.
\end{proof}

\begin{eg}\label{r=1}
For $r=1$ and $m_1=0$ (base point-free maps to $\bb{P}^1$), the dimension $\rho(g,r,d,m_0, 0)$ is nonnegative if and only if $d \geq \frac{g+m_0}{2},$ and the class $W^r_d(m_0, 0)$ is positive whenever $g> \rho(g,r,d)\neq 0.$
\end{eg}
\begin{proof}
The expected dimension is $g-2(g+1-d)-(m_0-1)+1=-g+2d-m_0 \geq 0$ if and only if $2d \geq g+m_0.$

We have 
$$\small{W^1_d(m_0, 0) =\theta^{2(g-d+1)+m_0-2} \frac{(m_0)(m_0+g+1-d)}{(m_0+g+1-d)!(g+1-d)!}\Big[(d+(m_0-1)(g-1))|m_0-1|-}$$

\small{\[(m_0-1)(m_0+g-d)|m_0-2|-0\Big]=\theta^{2(g-d+1)+m_0-1} \frac{(m_0)(m_0+g+1-d)}{(m_0+g+1-d)!(g+1-d)!}$$
$$\Big[(d+(m_0-1)(g-1))(m_0-1)-(m_0-1)(m_0+g-d)(m_0-2)\Big]$$ $$
= \theta^{2(g-d+1)+m_0-1} \frac{(m_0)(m_0+g+1-d)}{(m_0+g+1-d)!(g+1-d)!}(m_0-1)[(d+g-1)(m_0-1)-(m_0+g)(m_0-2)]$$
$$=\theta^{2(g-d+1)+m_0-1} \frac{(m_0)(m_0+g+1-d)}{(m_0+g+1-d)!(g+1-d)!}(m_0-1)(dm_0-m_0^2+g+1-d+m_0).\]}
Since $d \geq m_0,$ $dm_0 -m_0^2\geq 0.$  Since $2(g+1-d)$ must be between $0$ and $g$ we have $g+1-d \geq 0.$  Hence the total class is positive.
\end{proof}

\begin{eg}
For $r=2,$ when $g+r-d \geq 0$ and $\rho=0,$ the class $W^r_d(t,s,0)$ is zero precisely in the case above, and positive in all other cases.
\end{eg}

If $\rho=0,$ se can set $$g=\frac{1}{2}(3d-s-t-2).$$ Assuming that $t \neq s+1,$ the value of $W^2_d(t,s,0)$ is the positive factor $$\frac{g!}{(t+g+2-d)!(s+g+2-d)(g+2-d)}$$ times $$t(t+g+2-d)[(d+(g-1)(t-1))(t-1-s)-(t-1)(t+g+2-d-1)(t-2-s)-s(s+g+2-d)(t-s)]$$ $$+s(s+g+2-d)[(d+(g-1)(s-1))(t-s+1)-(s-1)(s+g+2-d-1)(t-s+2)-t(t+g+2-d)(t-s)].$$  
To show that this function is nonnegative, extend it to a function of real variables.  By computing the partial derivatives, one can show that this function is strictly increasing in $d$ for fixed $s$ and $t,$ and strictly increasing in $(t-s)$ for fixed $g$ and $d.$  It then remains to compute the minimal cases by hand and check that they are positive.

\begin{eg}
Values of $W^2_d$ for $r=2,$ $\rho=0$ and small values of $g$ and $d$ are included in Table 1.
\end{eg}

\begin{table}[ht]
\caption{Some Small Values of $W^2_d(0,s,t)$}
\label{Table1}
\begin{center}
\medskip
\begin{tabular}{rrrrr|rrrrr}
\\
d=g+2\\
g & d & s & t & $W^2_d(0, s, t)$ & g & d & s & t& $W^2_d(0,s,t)$\\
& & & & & \\
t=s+1 & & &   &                  & t=s+2 & & &  & \\
1 & 3 & 2 & 3 & 0                & 2 & 4 & 2 & 4 & 6\\
3 & 5 & 3 & 4 & 0                & 4 & 6 & 3 & 5 & 24\\
5 & 7 & 4 & 5 & 0                & 6 & 8 & 4 & 6 & 90\\
7 & 9 & 5 & 6 & 0                & 8 & 10 & 5 & 7 & 336\\
& & & & &\\
t=s+3 & & &   &                  & t=s+4\\
3 & 5 & 2 & 5 & 24               & 4 & 6 & 2 & 6 & 60\\
5 & 7 & 3 & 6 & 120              & 6 & 8 & 3 & 7 & 360\\
7 & 9 & 4 & 7 & 504              & 8 & 10 & 4 & 8 & 1680\\
&&&&&\\
t=s+5 & & &   &                  & t=s+6\\
5 & 7 & 2 & 7 & 120              & 6 & 8 & 2 & 8 & 210\\
7 & 9 & 3 & 8 & 840              & 8 & 10 & 3 & 9 & 1680\\
&&&&&\\
d=g+1\\
g & d & s & t & $W^2_d(0, s, t)$ & g & d & s & t& $W^2_d(0,s,t)$\\
&&&&&\\
t=s+1 & & &   &                  & t=s+2 & & &  &\\
4 & 5 & 2 & 3 & 24               & 3 & 4 & 1 & 3 & 24\\
6 & 7 & 3 & 4 & 240              & 5 & 6 & 2 & 4 & 240\\
8 & 9 & 4 & 5 & 1680             & 7 & 8 & 2 & 5 & 1680\\
&&&&&\\
t=s+3 & & &   &                  & t=s+4 & & &  &\\
4 & 5 & 1 & 4 & 120              & 5 & 6 & 1 & 5 & 360\\
6 & 7 & 2 & 5 & 1080             & 7 & 8 & 2 & 6 & 3360\\
8 & 8 & 3 & 6 & 7056             & 9 & 10 & 3 & 7 & 22176\\
&&&&&\\
t=s+5 & & &   &                  & t=s+6 & & &  &\\
6 & 7 & 1 & 6 & 840              & 7 & 8 & 1 & 7 & 1680\\
8 & 9 & 2 & 7 & 8400             & 9 & 10 & 2 & 8 & 18144\\
&&&&&\\
t=s+7 & & &   &                  & t=s+8 & & &  &\\
8 & 9 & 1 & 8 & 3024             & 9 & 10 & 1 & 9 & 5040\\

\end{tabular}
\end{center}
\vspace{-0.75\baselineskip}
\end{table}

Finally, consider the case when $g+r-d<0.$  Then the condition for a $g^r_d$ to exist is vacuous.  Indeed, if $m_i-(r-i)+g+r-d<0,$ then the condition for a $g^r_d$ to have an $(i+1)$-dimensional family of sections that vanish to order $m_i$ is vacuous.  So it is sufficient to apply the Porteous formula to those conditions that are not vacuous.  We have thus obtained the following theorem:

\begin{thm} The class $W^r_d(m_0, \cdots, m_r)$ is \small{\[\theta^c \sum \frac{(m'_k)(m_k+g+r-d)}{\prod (m_i+g+r-d)!}\Bigg[(d+(m'_k-1)(g-1))\prod_{i>j, \, i\neq k, \, j\neq k}(m_i-m_j)\prod_{i \neq k}|m_i-m_k+1|\]
\[- (m'_k-1)(m_k+g+r-d-1)\prod_{i>j, \, i\neq k, \, j\neq k}(m_i-m_j)\prod_{i<k}(m_i-m_k+2)\prod_{i>k}(m_k-m_i-2)\]
\[-\sum_{l \neq k}(m'_l)(m_l+g+r-d)|m_k-m_l|\prod_{i> j, \, i \neq k, \, i \neq l, \, j \neq k, \, j \neq l}(m_i-m_j)\prod_{i \neq k}|m_i-m_k+1||m_i-m_l+1|\Bigg],\]} where $$c=\sum_{(m_i-i+g+r-d)\geq 0} (m_i-i+g+d-r)$$ and all sums and products are defined over the non-vacuous multiplicities, where $$m_i-i+g+r-d \geq 0.$$
\end{thm}

This theorem allows us to calculate the set of equivalence classes of line bundles that give rise to $g^r_d$'s with the required ramification.  
But any class $[\s{L}]\in \Pic^d_C$ gives rise to a whole family of $g^r_d$'s when $g+r-d<0.$  To calculate the actual dimension of the family of $g^r_d$'s, we need to calculate the dimension of the class $\s{G}^r_d(m_i)$ on the Grassmann bundle $\s{G}^r_d=\bb{G}(r+1, \s{E})$ of $(r+1)$-dimensional spaces of sections of $H^0(\s{L}(nP)).$  

Let $\pi$ be the projection map from the Grassmann bundle $\s{G}^r_d$ to $\Pic^d_C.$  The fibers of the universal subbundle $\s{S}$ are our candidate $g^r_d$'s.  We still need to impose rank conditions such that the kernel of the map $\s{S} \rar \s{F}_i$ should have rank $i+1,$ so we set $r_i=r-i.$  The rank $b_i$ of $\s{F}_i$ is still $n+m_i,$ and the rank $a_i$ of $\s{S}$ is $r+1.$ So when we apply the filtered Porteous formula again, we have $\mu_i=n+ m_i+i-r.$   We need to calculate the Chern classes of $\s{S}.$  
Consider the exact sequence $$0 \rar \s{S} \rar \pi^*\s{E} \rar \s{Q} \rar 0.$$  So $c(\s{S}) \cdot c(\s{Q}) = c(\pi^*\s{E}).$  Thus $c(\s{S})=c(\pi^*\s{E}) \cdot c(\s{Q})^{-1}.$  The total Chern class $c(\s{Q})$ of the universal quotient is $1+\sigma_1 + \dots + \sigma_k,$ where $k$ is the rank of the quotient $\s{Q},$ in our case $n+d-r-g.$

Hence $$c(\s{F}_i-\s{S})=e^\theta \cdot \left(1+m'_i\left(d+(m'_i-1)(g-1)\right)\zeta +m'_i \gamma - m'_i(m'_i-1)\zeta\theta\right)\cdot (1+\dots+\sigma_k).$$   If $\rho_+ \leq 0,$ then every term in this determinant will contain higher powers of $\theta$ than $\theta^g,$ so it will have to vanish.  If $\rho_+ \geq 0,$ then $\s{G}^r_d(m_0, \cdots, m_r)$ is a sum of Vandermonde determinants in $\theta$ and the Schubert classes $\sigma_k$.

\section{Finiteness and Non-Existence Results}\label{LLS}

Let $(g,r,d, m_0, \cdots, m_r)$ be positive integers.  The moving-point Brill-Noether number is $$\rho=1+g-(r+1)(g+r-d)-\sum_{i=0}^r m_i-i.$$  Choose $(g, r, d, m_0, \cdots, m_r)$ such that $\rho \leq 0.$  We wish to prove that a general curve of genus $g$ admits at most finitely many $g^r_d$'s with vanishing sequence $(m_i)$ at any point.

We can degenerate our linear systems on a general curve to limit linear systems in the sense of Eisenbud and Harris (\cite{EH-GP} and \cite{EH-LS}), on a special reducible curve $X_0.$  

For the reducible curve $X_0$ we use the \emph{flag curve}, a semistable version of the $g$-cuspidal curve; it consists of a backbone chain of $g$ rational curves with an elliptic tail attached to a smooth point of each curve.  We may blow up the nodes, attaching extra rational curves.  The resulting curve $X_0$ looks like Figure 1.

Label the backbone curves by $Z_i,$ the "branches" connecting the elliptic curves to the backbone by $W^j_k,$ and the elliptic tails by $E_j.$  Label the intersections of $Z_i$ with $Z_{i+1}$ by $R_i$ and the nodes of the $W^j_k$ by $P_i.$

Let $X$ be a family of curves of genus $g,$ specializing to the flag curve $X_0.$  Let $(\s{L}, V)$ be a $g^r_d$ on the smooth fiber, possessing a ramification point with vanishing sequence $(m_0, \cdots, m_r).$   Assume that the ramification point specializes to a smooth point $Q.$  If this is not the case, we can always blow up the nodes; the result will still be a flag curve with more rational components.  Then we ask what are the possible limits of $\s{L}$ on $X_0.$

Let $\{(\s{L}_Y, V_Y)\}_{Y \text{ components of } X_0}$ be a limit linear series on $X_0,$ and let $Q$ be a ramification point of $(\s{L}_Y, V_Y)$ with ramification numbers $(m_i).$  We will try to count the possible points $Q.$

\begin{prp}\label{Qell}
If $\rho \leq 0,$ the limit of the ramification point $Q$ lies on one of the elliptic tails.
\end{prp}

\begin{proof}
The limit is a smooth point.  So it can not be one of the nodes of a rational curve.  But any smooth point on a rational curve has weight at most $$(r+1)(d-r)-r(g-1).$$  In our case, however, if $\rho \leq 0,$ then $$w \geq (d-r)(r+1)-rg-1 > (d-r)(r+1)-r(g-1).$$
\end{proof}

For all $i,$ we have $i \leq m_i(V_{Z_l},R_l)\leq d-r+i.$  So $$(r+1)(d-r) \geq \sum_{i=0}^r m_i(V_{Z_N}, R_N)-m_l(V_{Z_2}, R_2)$$ since there are $r+1$ terms in the sum, each at most $d-r.$

We have $$m_i(V_{Z_{l+1}},R_{l+1})\geq m_i(V_{Z_l}, R_l),$$ and if $Z_l$ meets one of the $g$ tails, then for all but all but $1$ value of $i,$ $$m_i(V_{Z_{l+1}},R_{l+1})>m_i(V_{Z_l}, R_l).$$  So for these $Z_l,$
\begin{equation}\label{dropbyr}
\sum_{i=0}^r m_i(V_{Z_{l+1}},R_{l+1})-m_i(V_{Z_l}, R_l)\geq r.
\end{equation}

We apply these inequalities to bound the ramification from above and below on the elliptic curve containing $Q$, and to identify the possible $g^r_d$'s.

\begin{prp}
Let $P_0$ be the intersection of a backbone curve $Z_i$ with the $j^{th}$ chain curve $W^1_j.$  Then the weight of $V_{Z_i}$ at $P_0$ is at most $(r+1)(d-r)-r(g-1).$
\end{prp}

\begin{proof}
Since the total ramification on a rational curve is $(r+1)(d-r),$ and the weights at $R_i$ and $R_{i+1}$ add up to $r(g-1),$ we have the weight $$w(V_{Z_i},P_0) \leq (r+1)(d-r)-r(g-1).$$  
\end{proof}

\begin{prp}[Minimum Weight at P]\label{minP}
Let $P_k$ be the intersection of a chain curve $W^k_j$ with the next chain curve $W^{k+1}_j$ or with the elliptic tail $E_j.$  Then the weight $w(V_{W^{k+1}_j},P_k)$ or $w(V_{E_j}, P_k)$ is at least $r(g-1).$
\end{prp}

\begin{proof}
 For $0< k \leq n,$ let $P_k$ be the intersection of $W^k_j$ with $W^{k+1}_j.$  The proof is by induction on $k.$  Since $w(V_{Z_i},P_0) \geq (r+1)(d-r)-r(g-1),$ and by the Compatibility Condition $$w(V_{Z_i},P)+w(V_{W^1_j},P_1)=(r+1)(d-r),$$ we have $$w(V_{W^1_j},P_0) \geq r(g-1).$$  For the induction step, since the total ramification on a rational curve is $(r+1)(d-r)$ by the Pl\"ucker formula, we have $$w(V_{W^k_j}, P_k) \geq (r+1)(d-r)-r(g-1).$$  Hence, by applying the Compatibility Condition, $$w(V_{W^{k+1}_j},P_k) \geq r(g-1).$$
\end{proof}

\begin{prp}[Maximum Weight at P]\label{maxP}
Let $Q$ lie on the elliptic tail $E_j.$  The vanishing sequence of $V_{E_j}$ at its node $P$ is at most $(d-m_{r-i}),$ and $$w(V_{E_j},P) \leq (r+1)d-\sum m_i-\frac{r(r+1)}{2}.$$  If $\rho$ is the moving-point Brill-Noether number $$\rho=1+g-(r+1)(g+r-d) - \sum_{i=0}^r m_i + \frac{r(r+1)}{2},$$ then $$w(V_{E_j}, P) \leq r(g+1)+\rho+r-1.$$
\end{prp}

\begin{proof}
Since the sum $$m_i(V_{E_j}, Q)+m_{r-i}(V_{E_j}, P) \leq d,$$ the vanishing sequence at $P$ is at most $(d-m_i).$  Sum the multiplicities to get $$w(V_{E_j}, P) \leq (r+1)d-\sum m_i -\frac{r(r+1)}{2}.$$  But we have $$\rho=1-rg-r(r+1)+(r+1)d-\sum m_i +\frac{r(r+1)}{2}=1-rg+w(V_{E_j}, P).$$  So $$w(V_{E_j}, P) \leq r(g+1)+\rho+r-1.$$
\end{proof}

Combining the minimum and maximum conditions, we obtain the following:

\begin{thm}[Non-Existence For Sufficiently Low $\rho$]
Let $\rho$ be the \emph{Moving-Point Brill-Noether number} $$\rho=1+g-(r+1)(g+r-d) - \sum_{i=0}^r (m_i -i).$$ 

If $\rho<1-r,$ then there is no $g^r_d$ on a general curve of genus $g$ with vanishing sequence $(m_i)$ at any point $Q.$
\end{thm}

\begin{prp}[Finiteness of Points for $\rho \leq 0$]
If $\rho \leq 0,$ then there are at most finitely many points $Q$ for which a $g^r_d$ exists with multiplicities $m_i$ at $Q.$
\end{prp}

\begin{proof} 
As in the previous proof, the limit of any such point $Q$ must lie on an elliptic tail.  The flag curve only has finitely many elliptic tails, so it's enough to show that on any one tail $E,$ there are only finitely many possible limiting points $Q.$  We bound the weights at $P;$ $$r(g-1) \leq w(V_{E_j},P)\leq r(g-1)+\rho+r-1.$$  So the difference between the maximal and minimal possible weights is $\rho+r-1.$  Since $r\leq 0,$ this is at most $r-1.$  Therefore, since there are $r+1$ places in the multiplicity sequence and they differ by only $r-1,$ there are at least two positions $i$ and $j$ where $m_{r-i}(P)$ is exactly the maximum value $d-m_i(Q)$ and $m_{r-j}(P)$ is exactly the maximum value $d-m_j(Q).$  
Thus the linear system contains divisors $m_iQ + (d-m_i)P$ and $m_jQ + (d-m_j)P.$

So $Q-P$ must be $(m_i-m_j)$-torsion.  Hence there are at most finitely many possible choices for $Q$.
\end{proof}

\begin{rmk}
Note that the finiteness of points implies the Brill-Noether non-existence theorem with a fixed general ramification point: if $\rho_{\text{ fixed } Q}<0$ then we have $\rho_{\text{ moving } Q} \leq 0.$  So there are only finitely many $Q$ possessing a $g^r_d$ with ramification $(m_0, \cdots, m_r)$.  In particular, a general $Q$ does not possess such a $g^r_d.$
\end{rmk}

\begin{thm}[Finiteness and Non-Existence of Linear Systems for $r=1,$ $\rho \leq 0$]\label{r1}
If $r=1$ and the expected dimension is $\rho(g, 1, d, m_0, m_1)=0,$ then a general curve of genus $g$ possesses at most finitely many $g^1_d$'s with a ramification point of type $(m_0, m_1).$  If $\rho<0,$ then no such $g^1_d$'s exist.
\end{thm}

\begin{proof}
The limiting ramification point $Q$ must land on a torsion point of an elliptic component $E$, and the aspect of the limit $g^1_d$ on that elliptic component becomes $m_0Q+(d-m_0)P, m_1Q + (d-m_1)P.$  But we need to count the complete limit linear series, not just their $E$-aspects.  

Since the aspect on $E$ is $m_0Q+(d-m_0)P, m_1Q+(d-m_1)P,$ by the Compatibility Condition, the $(X_0-E)$-aspect must have ramification $(m_0, m_1)$ at $P.$  We calculate the dimension of the family of $g^1_d$'s on $(X_0-E)$ with a fixed ramification point of type $(m_0, m_1).$  Since $(X_0-E)$ consists of rational and elliptic curves, they are all general.  There are at most three nodes on the rational components and only one on the elliptic components, so the nodes are all general points (since there is an automorphism that replaces these nodes with any others), and hence $(X_0-E)$ satisfies the Additivity Condition.  Hence the dimension of possible $g^r_d$'s on $(X_0-E)$ with a fixed ramification point at $P$ of type $(m_0, m_1)$ is $$\rho_{\text{fixed}}(g-1, 1, d, P, m_0, m_1)=(g-1)-2(g-1+1-d)-m_0-m_1+1 $$ $$= g+2(g+1-d)-1+2-m_0-m_1+1= \rho(g,1,d, m_0, m_1).$$  So if $\rho=0$ there are finitely many, and if $\rho<0$ there are none.  Since there are only finitely many possible choices for $E$ and finitely many choices for $(X_0-E)$, there are a total of finitely many possible limit linear series with this ramification, and therefore a total of finitely many possible $g^r_d$'s on the general curve.
\end{proof}

When $r=2,$ we can not always prove non-existence for $\rho=-1,$ but we can still prove finiteness when $\rho=0.$

\begin{thm}[Finiteness of Linear Systems for $r=2,$ $\rho\leq 0$]
If $r=2$ and the expected dimension is $\rho(g,2,d, m_i) \leq 0,$ then there are at most finitely many $g^2_d$'s on a general curve of genus $d$ that possess a ramification point with vanishing sequence $(m_0, m_1, m_2).$  
\end{thm}

\begin{proof}
As before, we can degenerate the curve to the flag curve.  The limiting position of the ramification point $Q$ is a torsion point on an elliptic tail $E$ relative to the node $P.$  The difference between the minimum and maximum possible weights of the $E$-aspect at the node $P$ is at most $r-1=1.$  So the linear system on $E$ is generated by three divisors, at least two of which are linear combinations of $P$ and $Q$ exclusively.

If $(m_2-m_1)$ and $(m_1-m_0)$ are relatively prime, then the three divisors $$m_0Q+(d-m_0)P, \text{ } m_1Q+(d-m_1)P \text{ and } m_2Q+(d-m_2)P$$ can not all be linearly equivalent, since that would require that $Q=P.$  So the linear system can only be of the form $m_0Q+(d-m_0)P,$ $m_1Q+(d-m_1)P$ and $m_2Q+(d-m_2-1)P+R,$ up to renumbering the $m_i$'s.  The point $R$ is completely determined by the linear equivalence.  So there are only finitely many such aspects on $E.$  Since the ramification of the $E$ aspect at $P$ is $(d-m_0, d-m_1, d-m_2-1),$ the ramification on $X_0$ at $P$ is $(m_0, m_1, m_2+1).$  We can compute the dimension of possible $g^2_d$'s on the complement $(X_0-E)$ with this ramification at the fixed point $P:$ 
$$(g-1)-3(g-1+2-d)-m_0-m_1-m_2-1+3=g-3(g+2-d)+2 - m_0-m_1+3 = \rho.$$

In case $(m_2-m_1)$ and $(m_1-m_0)$ have a common factor, then there is also the possibility that the $E$-aspect is just $$m_0Q+(d-m_0)P, \text{ } m_1Q+(d-m_1)P, \text{ } m_2Q+(d-m_1)P.$$  In this case, the ramification of the $E$-aspect is $(d-m_0, d-m_1, d-m_2)$ at $P$, so the ramification of the $(X_0-E)$-aspect is only $(m_0, m_1, m_2).$  The dimension of the family of such limit linear series is $1.$

Suppose that the general curve of genus $g$ actually had a $1$-parameter family of $g^2_d$'s with ramification $(m_0, m_1, m_2).$  Consider the class $[\Lambda]$ of this locus in the Grassmann bundle $\s{G}^r_d.$  If it is actually a non-empty locus of dimension $1,$ then its class is $a\theta^{g-1}\sigma_{\text{top}} + b\theta^g\sigma_{\text{top}-1},$ where $\sigma_{\text{top}}$ is the top Schubert class, for some nonnegative coefficients $a$ and $b.$  Then we should be able to intersect it with the codimension $1$ class $\lambda$ of linear series that are ramified at a fixed general point $R.$  This class is of the form $c \theta + e \sigma_1.$  Assume that the rank of $\s{E}$ is at least $4,$ which we can force by choosing $n$ sufficiently large.   Then the coefficient $e$ is nonzero, since the intersection with the fiber over any point of $\Pic^d_C$ is non-empty: if the line bundle $\s{L}(nP)$ has a $4$-dimensional family of sections, then we can certainly pick a $3$-dimensional subfamily that vanish to orders at least $(0, 1, 3)$ at $R.$  But the intersection of $\sigma_1$ with any class is positive.  Hence $\lambda \cap \Lambda$ is positive.  

Hence there must exist a non-empty family $L$ of $g^2_d$'s with ramification $(m_0, m_1, m_2)$ at $Q$ and at least simple ramification at $R,$ for \emph{every} fixed point $R$ on the general curve.  But what happens when we try to degenerate these $g^2_d$'s to $X_0?$  Since the condition for a point $R$ to be a ramification point of a $g^2_d$ is a closed condition, it must be that every point $R$ on $X_0$ is a ramification point of some limit linear series. But there is only one possibility for the $E$-aspect, and it can only be ramified at finitely many points.  At a fixed general point $R$ on $E,$ there is no ramification.  Hence we obtain a contradiction.

So there can be at most finitely many $g^2_d$'s with a ramification point of type $(m_0, m_1, m_2).$
\end{proof}

The only remaining open question on maps to the plane is whether there exist $g^2_d$'s with ramification of expected dimension $-1$ when the ramification numbers have a common factor.  We have constructed such $g^2_d$'s on the flag curve, and indeed they exist on any reducible curve containing an elliptic component, but they need not exist on the general curve. We shall see in Section \ref{PC} that in fact they do not.

When $r=3,$ the situation becomes more complicated and begins to resemble the general case.

\begin{prp}[Finiteness Condition for $r=3$]
If $r=3$ and the expected dimension is $\rho(g,r,d, m_i)\leq 0,$ then the dimension of $\s{G}^r_d(m_0, \cdots, m_3)$ is at most $1.$  If in addition, the differences $m_i-m_j$ are pairwise relatively prime, then $W^r_d$ is finite.
\end{prp}

\begin{proof}
Degenerate the curve to the flag curve $X_0,$and consider the possible vanishing sequences at the node $P$ on $E.$  As in the previous proofs, the vanishing sequence is bounded by $(d-m_i)$ and is allowed to differ from its maximum values by at most $r-1=2.$  We shall consider each possible ramification at $P.$

If all the pairwise differences among the multiplicities share a common factor, then the first possible $E$-aspect is simply $$m_0Q+(d-m_0)P, m_1Q+(d-m_1)P, m_2Q+(d-m_2)P, m_3Q+(d-m_3)P.$$  In this case we have finitely many $E$-aspects and a $2$-parameter family of possible $X_0-E$-aspects.  However, only finitely many of them can deform to the general curve of genus $g$ because otherwise at least finitely many would have to have ramification at a general fixed point $R,$ and in the limit there are only finitely many possible $E$-aspects and therefore only finitely many possible fixed ramification points on $E$.

If at least two of the pairwise differences share a common factor, then we could have an $E$-aspect of the form $$m_0Q +(d-m_0P), m_1Q+(d-m_1)P, m_2Q+(d-m_2)P, m_3Q+(d-m_3-1)P+R$$ for some point $R.$  We have finitely many $E$-aspects and a $1$-parameter family of possible $X_0-E$-aspects.  Or we could have $$m_0Q+(d-m_0)P,  m_1Q+(d-m_1)P, m_2Q+(d-m_2)P, m_3Q+(d-m_3-2)P+R+S,$$ for some effective divisor $R+S$ of degree $2.$  In this case there is a $1$-parameter family of possible $E$-aspects, since $R$ can be chosen arbitrarily and then $S$ is determined, but we are imposing a fixed point with vanishing sequence $(m_0, m_1, m_2, m_3+2)$ on $Y,$ so there are only finitely many $Y$-aspects.  So these cases contribute a $1$ parameter family if the pairwise differences are not relatively prime.

Finally, if all the pairwise differences are relatively prime, then the only option is an $E$-aspect of the form $$m_0Q+(d-m_0)P, m_1Q+(d-m_1)P, m_2Q+(d-m_2-1)P+R, m_3Q+(d-m_3-1)P+S.$$  There are finitely many possible such aspects.  The corresponding $Y$-aspects have vanishing sequence $(m_0, m_1, m_2+1, m_3+1)$ at $P,$ so there are finitely many of them as well.  Hence if the pairwise differences are relatively prime, then there are only finitely many $g^3_d$'s with the specified ramification type.
\end{proof}

If $r\geq 4,$ then we never have all the pairwise differences relatively prime, since at least two of them are even.  However, we can still prove a bound on the dimension.

\begin{thm}[Weak General Bound]\label{WkBd}
The dimension of $\s{G}^r_d(m_i)$ over a general curve of genus $g$ is bounded by $\rho+r-2$ if this number is nonnegative.  Moreover, let $k+1$ be the size of the largest subset of the set of multiplicities $\{m_{i_0}, \cdots, m_{i_k}\}\subseteq\{m_0, \cdots, m_r\}$ whose pairwise differences all share a common factor.  Then the dimension of $\s{G}^r_d(m_0, \cdots, m_r)$ is bounded by $\rho+k-1.$
\end{thm}

\begin{proof}
As before, if we degenerate the curve to a flag curve.  Since all points are general on a rational curve, if the limit of the ramification point $Q$ lies on a rational component, then by the additivity theorem the dimension of $\s{G}^r_d$ is just the fixed-point ramification number $\rho-1.$  So suppose that the limit of $Q$ on $X_0$ lies on one of the elliptic tails.  Then it is in fact a torsion point.  We have the upper and lower bounds $$r(g-1) \leq w(V_{E},P)\leq r(g-1)+\rho+r-1.$$  

The multiplicities of $V_E$ at $P$ are allowed to be equal to their maximum values at the $k+1$ places whose pairwise differences have a common factor.  The multiplicities at the other $r-k$ places are required to drop by $1$ because $Q \neq P.$  So the difference between the actual lower and upper bounds on $w(V_{E}, P)$ is $\rho+k-1.$  If $\rho+k-1<0,$ then there are no possible $g^r_d$'s.  Assuming this difference is nonnegative, we can distribute it between $E$ and $X_0-E.$

Let $t$ be any integer between $0$ and $\rho+k-1.$  Then we can construct an $E$-aspect of the form $$m_0Q+(d-m_0)P, \cdots, m_kQ+(d-m_k)P, m_{k+1}Q+(d-m_{k+1}-1-t)P+D_{k+1},$$ $$ m_{k+2}Q+(d-m_{k+2}-1)Q+D_{k+1}, \cdots, m_rQ+(d-m_r-1)P+D_r,$$ where the $D_i$ are effective divisors of degree $d_i$ whose sum is $t+r-k.$  There is a $t$-parameter family of such aspects.  The corresponding $(X_0-E)$-aspects must have multiplicity sequence $$(m_0, \cdots, m_k, m_{k+1}+d_{k+1}, m_{k+2}+d_{k+2}, \cdots, m_r+d_r)$$  There is a $(\rho+k-1-t)$-parameter family of such $(X_0-E)$-aspects. Thus in every case, there is a $(\rho+k-1)$-parameter family of pairs of an $E$-aspect with a $X_0-E$-aspect.

However, in case $k=r,$ if all the pairwise differences have a common factor, the bound is only $\rho+r-2$ if this is nonnegative.  The reason is that if we subtract $t$ from $w(V_E, P),$ we only gain a $(t-1)$-parameter family because one point is determined by the others, and it is not possible to have $m_0Q+(d-m_0)P, \cdots, m_rQ+(d-m_r)P$ on $E$ and a $(\rho+r-1)$-dimensional family on $X_0-E$ because the resulting $g^r_d$'s would not be ramified at a general fixed point $R$ on $E.$
\end{proof}

\section{Plane Curves}\label{PC}

In the previous section we proved a finiteness condition for $r=2$ but could not prove the full dimensionality.  In this section we use the special extrinsic properties of plane curves to provide an ad hoc proof of the missing non-existence case when $r=2.$

\begin{thm}\label{r2}
Given nonnegative integers $g,$ $d,$ $(m_0>m_1>m_2),$ a general curve $C$ of genus $g$ admits at most a $\rho(g, 2, d, m_0, m_1, m_2)$-dimensional family of $g^2_d$'s with a ramification point of type $(m_0, m_1, m_2).$
\end{thm}

\begin{proof}

We have proved all cases except $\rho=0.$  Consider the case $\rho=0.$

If $m_2 > 0,$ then a $g^2_d$ with a point of type $(m_0, m_1, m_2)$ is equivalent to a $g^2_{d-m_2}$ with a point $Q$ of type $(m_0-m_2, m_1-m_2, 0),$ by subtracting off the basepoint.  This does not change $\rho.$  So it is sufficient to consider basepoint-free $g^2_d$'s with a point of type $(m_0, m_1, 0).$  Such $g^2_d$'s give rise to maps $\phi$ from $C$ to $\bb{P}^2.$  Consider the possible images of $C$ in $\bb{P}^2.$  

We first show that a general map $C \rar \bb{P}^2$ is the normalization map.  Suppose that the map factors as $\phi_1 \: C \rar C'$ and $\phi_2\: C' \rar \bb{P}^2,$ where $\phi_2$ is the normalization map.  By Hurwitz's Theorem, the degree of the ramification divisor of $\phi_1$ is $R=2g-2-(\deg \phi_1)(2g'-2).$ where $R$ is the degree of the ramification divisor.  A map of curves is determined by its ramification points.  So the dimension of the family of possible $\phi_1$ is $3g'-3+R=3g'-3+\deg(\phi_1)(2g'-2)+(2g-2) \leq (2g-2)<3g-3,$ assuming $g' \geq 2.$  If $g'=1,$ the dimension is $1+2g-2 \leq 3g-3.$  Since there are finitely many possible degrees and finitely many possible dimensions, the family of all curves mapping to a curve of genus between $1$ and $g$ is a proper subvariety of $\?{M_g}.$  So a general curve does not have such a map.  A general curve does map to $\bb{P}^1,$ so we have to check the dimension of possible factorizations $\phi_1\: C \rar \bb{P}^1$ and $\phi_2\: \bb{P}^1\rar \bb{P}^2.$  Suppose that $\phi_1$ has degree $d_1$ and vanishing orders $(0,c)$ at $Q,$ and $\phi_2$ has degree $\frac{d}{d_1}$ and vanishing orders $(0, \frac{m_1}{c}, \frac{m_2}{c}).$  For $\phi_1$ to exist, we need $\rho_1=g-2(g+1-d_1)-c+2=2d-g-c$ to be nonnegative.  So $d_2 \leq \frac{2d}{g+c}.$  To count the possible $\phi_2$'s, we have $\rho_2=0-3(2-d_2)-\sum(\frac{m_i}{c}-i) +1=3d_2-5-\sum(\frac{m_i}{c}-i) \leq \frac{6d}{g+c}-5-\sum(\frac{m_i}{c}-i),$ since $d_2 \leq \frac{6d}{g+c}.$  So $\rho-\rho_2=d(3-\frac{6}{g+c}-\sum(m_i-i)(1-\frac{1}{c}).$  No section can vanish to order higher than $d,$ and the map basepoint-free, so one section vanishes to order zero, so $\sum(m_i-i)\leq 2d-1.$  Hence $\rho-\rho_2 \geq d(1-\frac{6}{g+c}+\frac{2}{c}.$  If $g>1,$ this number is always positive.  So $\rho>\rho_2.$  Thus it is enough to assume that $\phi$ is the normalization map to a degree-$d$ plane curve of genus $g.$

Without any ramification, a degree-$d$ curve of genus $g$ has exactly $\frac{1}{2}(d-1)(d-2)-g$ nodes.  Choose each node to lie anywhere in $\bb{P}^2,$ gaining $2$ degrees of freedom, but the node is a double point that imposes $3$ conditions.  Each node represents a net loss of $1$ in the dimension.  Since the dimension of all degree-$d$ curves is $\frac{1}{2}(d)(d+3),$ and a set of up to $\frac{1}{2}(d)(d+3)$ point conditions always imposes independent conditions, the dimension of degree-$d$ curves of genus $g$ is thus $3d-1+g.$ Hence $g=\rho+3(g+2-d)+m_0+m_1-2.$  Hence the dimension is bounded above by $3d-1+\rho + 3g+6-3d+1+m_0+m_1-2$ or $3g+\rho+m_0+m_1.$ 

Since the dimension of $\s{M}_g$ is $3g-3,$ the ramification-free Brill-Noether number $g-3(g+2-d)$ is bounded by $\rho+m_0+m_1-4,$ and the linear changes of coordinate bases form an $8$-dimensional family of image curves for each $g^2_d$ on an abstract curve, this is consistent with the ramification-free Brill-Noether theorem: without ramification, the classical Brill-Noether number would be $\rho+m_0+m_1-4.$  The total space has dimension at most $3g+\rho + m_0+m_1,$ so it is impossible for every curve of genus $g$ to admit more than a $\rho+m_0+m_1-4$-dimensional family of $g^2_d$'s.

Next, impose the required ramification at $Q$ and watch how the dimension changes. 

The images must have an $m_1$-fold point at $\phi(Q).$  In local coordinates, the map $\phi\: C \rar \bb{P}^2$ looks like $t \rar (a_0 t^{m_0}+a_1 t^{m_0+1} +\cdots, b_0 t^{m_1}+b_1t^{m_1+1}+ \cdots).$  Apply resolution of singularities.  Let $m_0=q_1m_1+r_1.$  By successively blowing up $\phi(Q),$ we obtain a sequence of $q_1$ points where $\phi(Q)$ lifts to an $m_1$-fold point.  At the next blowup, it lifts to an $r_1$-fold point $(a_0 t^{m+1}+a_1 t^{m_1+1} +\cdots, b_0 t^{r_1}+b_1t^{r_1+1}+ \cdots).$  If $m_1=q_2r_1+r_2,$ we get $q_2$ $r_1$-fold points followed by an $r_2$-fold point.  At the last step of the resolution, if the greatest common divisor of $(m_0, m_1)$ is $1,$ the last blowup gives us an inflection point of type $(1,r_k).$  This resolves into $r_k-1$ successive fixed inflection points before we finally hit a simple point on the curve, transverse to the exceptional divisor.  If the greatest common divisor is not $1,$ we end up with an $r_k$-fold point of type $(r_k,r_k),$ so the map looks like $(a_0t^{r_k}+\cdots, t^{r_k}+\cdots)$ in coordinates.  Blowing this up, the map becomes $(a_0+a_1t+\cdots, b_0t^{r_k}+b_1t^{r_k+1}+\cdots).$  We can change coordinates to obtain an inflection point $(a_1t+\cdots, b_0t^{r_k}+\cdots)$ of multiplicity type $(1,r_k),$ almost as if we were blowing up a point of type $(r_k-1, r_k)$, but now we don't know where on this exceptional divisor the point lies!  We have reintroduced one extra degree of freedom.

Each new virtual $r$-fold point drops the genus by $\frac{1}{2} r(r-1),$ freeing up that many nodes, but it also imposes $\frac{1}{2} r(r+1)$ conditions.  So a virtual $r$-fold point is a net loss of $r$ dimensions.  Each virtual inflection point imposes a condition but leaves the genus alone, for a net loss of $1$ dimension.

So, by requiring the existence of a ramification point $Q,$ we gain two dimensions for the image of $Q$ itself, which is free to move in $\bb{P}^2,$ but if $gcd(m_0, m_1)=1,$ we lose $q_1(m_1)+q_2(r_1)+\cdots +q_k(r_k)+r_k-1.$  The telescoping sum can be rewritten as $(m_0-r_1)+(m_1-r_2)+ \cdots +(r_{k-2}-r_k)+(r_{k-1}-1)+r_k-1 =m+0+m_1-2.$  So we have lost $m_0+m_1-4$ dimensions.  If $gcd(m_0, m_1)\neq 1,$ then we obtain the telescoping sum $(m_0-r_1)+(m_1-r_2)+ \cdots +(r_{k-2}-r_k)+(r_{k-1}-0)+r_k-2=m_0+m_1-2,$ so again the dimension drops by $m_0+m_1-4,$ to a total dimension of $\rho+ 3g+4.$  Once again, it is impossible to have a $3g-3$-parameter family of fibers all of dimension at least $\rho+8$, when the total space dimension is down to $\rho+3g+4.$
\end{proof}

This proof does not generalize to higher dimensions; the genus of a curve in $\bb{P}^n$ is not determined by its singular points and their blowups.  But it has the advantage of generalizing to multiple (fixed or moving) ramification points. It also shows that there are counterexamples on the flag curve that do not deform to the general curve, so the existence of counterexamples on the flag curve in higher dimensions should also not be seen as strong evidence against the dimensionality conjecture for the general curve.  It also shows that although the flag curves are "Brill-Noether general" for $g^r_d$'s without ramification and with fixed ramification points, they are not sufficiently general when movable ramification points are imposed.  This suggests that the Brill-Noether loci on the moduli spaces $\overline{M_g}$ for these ramification conditions may well be different from the known classical Brill-Noether loci and the loci for fixed general ramification points. Some of these loci will be the subject of a future paper.

\end{document}